\documentclass[11pt,a4paper]{article}
\usepackage{caption}
\usepackage{epsf,epsfig,amsfonts,amsgen,amsmath,amstext,amsbsy,amsopn,amsthm}
\usepackage{color}
\usepackage{mathrsfs}
\usepackage{tikz}
\usepackage{enumerate}
\usepackage{url}
\usepackage{enumitem}
\usepackage{amsmath}
\usepackage{amssymb}
\usepackage{amsthm}
\usepackage{cases}
\usepackage{txfonts}
\newtheorem{thm}{Theorem}[section]

\newtheorem{claim}[thm]{Claim}
\newtheorem{lemma}[thm]{Lemma}

\newtheorem{lem}[thm]{Lemma}

\newtheorem{conjecture}[thm]{Conjecture}

\def\QED{\hfill \rule{7pt}{7pt}}

\newcommand{\dA}{{\mathcal{A}}}
\newcommand{\dG}{{\mathcal{G}}}
\newcommand{\dN}{{\mathbb{N}}}
\newcommand{\dR}{{\mathcal{R}}}
\newcommand{\dE}{{\mathcal{E}}}
\newcommand{\dM}{{\mathfrak{M}}}
\newcommand{\dm}{{\mathfrak{m}}}
\newcommand{\dn}{{\mathfrak{n}}}
\newcommand{\dJ}{{\mathfrak{J}}}
\newcommand{\e}{{\textbf{e}}}

\newcommand{\RNum}[1]{\uppercase\expandafter{\romannumeral #1\relax}}

\newcommand{\1}{{\uppercase\expandafter{\romannumeral1}}}
\newcommand{\2}{{\uppercase\expandafter{\romannumeral2}}}
\begin{document}
\title{Improvements on induced subgraphs of given sizes}

\author{
Jialin He\thanks{School of Mathematical Sciences, USTC, Hefei, Anhui 230026, China. Email: hjxhjl@mail.ustc.edu.cn.}
\and
Jie Ma\thanks{School of Mathematical Sciences, USTC, Hefei, Anhui 230026, China. Email: jiema@ustc.edu.cn.
Research supported in part by NSFC grant 11622110, National Key Research and Development Project SQ2020YFA070080, and Anhui Initiative in Quantum Information Technologies grant AHY150200.}
\and
Lilu Zhao\thanks{School of Mathematics, Shandong University, Jinan 250100, China. Email: zhaolilu@sdu.edu.cn.
Research supported in part by NSFC grant 11922113.}
}
\date{}
\maketitle

\begin{abstract}
Given integers $m$ and $f$, let $S_n(m,f)$ consist of all integers $e$ such that every $n$-vertex graph with $e$ edges contains an $m$-vertex induced subgraph with $f$ edges,
and let $\sigma(m,f)=\limsup_{n\rightarrow\infty} |S_n(m,f)|/\binom{n}{2}$.
As a natural extension of an extremal problem of Erd\H{o}s, this was investigated by Erd\H{o}s, F\"uredi, Rothschild and S\'os twenty years ago.
Their main result indicates that integers in $S_n(m,f)$ are rare for most pairs $(m,f)$,
though they also found infinitely many pairs $(m,f)$ whose $\sigma(m,f)$ is a fixed positive constant.
Here we aim to provide some improvements on this study.
Our first result shows that $\sigma(m,f)\leq \frac12$ holds for all but finitely many pairs $(m,f)$ and the constant $\frac12$ cannot be improved. This answers a question of Erd\H{o}s et. al.
Our second result considers infinitely many pairs $(m,f)$ of special forms,
whose exact values of $\sigma(m,f)$ were conjectured by Erd\H{o}s et. al.
We partially solve this conjecture (only leaving two open cases)
by making progress on some constructions which are related to number theory.
Our proofs are based on the research of Erd\H{o}s et. al and involve different arguments in number theory.
We also discuss some related problems.
\end{abstract}


\section{Introduction}
The {\it Tur\'an number} $ex(n,H)$ of a graph $H$ denotes the maximum number of edges in an $n$-vertex graph which does not contain $H$ as a subgraph.
Since the seminal work of P. Tur\'an, the study of Tur\'an numbers has been a central theme in extremal graph theory (see the survey \cite{FS13}).
A natural generalization, which was proposed by Erd\H{o}s \cite{E64} in 1963 to reduce the structure of forbidden subgraphs to one parameter (namely, their size),
asks the maximum number of edges in an $n$-vertex graph where every $m$-vertex subgraph spans less than $f$ edges.\footnote{Here and throughout this paper, $m$ and $f$ are integers satisfying $m\ge 2$ and $0\le f\le \binom{m}{2}$.}
This density problem and its notorious hypergraph version (initialed in \cite{BES}) are related to difficult problems in number theory (e.g. the work of Ruzsa-Szemer\'edi \cite{RS} on Roth's theorem \cite{Roth}) and remain unsolved in general.

Here we consider another natural extremal problem, which can be viewed as an ``induced subgraph'' analogue of the above Erd\H{o}s' problem in \cite{E64}.
This was first investigated by Erd\H{o}s, F\"uredi, Rothschild and S\'os \cite{EFRS}.
Following their notation, we say $(n,e)\rightarrow(m,f)$ if every $n$-vertex graph with $e$ edges contains an induced $m$-vertex subgraph with exactly $f$ edges.
Taking an example, if $t_p(n)$ denotes the number of edges in the complete balanced $p$-partite graph on $n$ vertices,
then Tur\'an's theorem can be equivalently stated as that $(n,e)\rightarrow(m,\binom{m}{2})$ if and only if $e>t_{m-1}(n)$.
For a fixed pair $(m,f)$, let $S_n(m,f)=\{e: (n,e)\rightarrow(m,f)\}$ and let
\begin{equation}\label{equ:sigma}
\sigma(m,f)=\limsup_{n\rightarrow\infty}\frac{|S_n(m,f)|}{\binom{n}{2}}.
\end{equation}
Since $S_n(m,f)$ is a subset of $\left\{0,1,...,\binom{n}{2}\right\}$ which cannot contain $0$ and $\binom{n}{2}$ simultaneously,
the fraction on the right hand of \eqref{equ:sigma} is at most $1$ and thus any pair $(m,f)$ satisfies $0\leq \sigma(m,f)\leq 1$.
In \cite{EFRS}, Erd\H{o}s, F\"uredi, Rothschild and S\'os gave a number of constructions arising from extremal graph theory,
which reveal that for most pairs $(m,f)$, the pairs $(n,e)$ satisfying $e\in S_n(m,f)$ are relatively rare (in sense of the measure $\sigma(m,f)$).
Their main result is as follows.
Throughout the rest, let $\mathcal{A}= \{(2,0),(2,1),(4,3),(5,4),(5,6)\}$.

\begin{thm}[\cite{EFRS}]\label{Thm:EFRS}
If $(m,f)\notin \mathcal{A}$, then $\sigma(m,f)\leq \frac23$;
otherwise, $\sigma(m,f)=1$.
\end{thm}

Motivated by this result, the authors \cite{EFRS} raised the question if
\begin{equation}\label{equ:1/2}
\mbox{``$\sigma(m,f)>\frac12$ holds for only finitely many pairs".}
\end{equation}
Along the way to prove Theorem~\ref{Thm:EFRS}, an interesting intermediate result in \cite{EFRS} (see Construction 8 therein) says that the majority of the pairs $(m,f)$ satisfy $\sigma(m,f)=0$.
On the other hand, they \cite{EFRS} also proved the following ``positive'' result by bounding $\sigma(m,f)$ below by a positive constant for infinitely many pairs $(m,f)$.

\begin{thm}[\cite{EFRS}]\label{Thm:Non zeros}
Let $m, f$ be integers such that there exist positive integers $a,b,c$ satisfying $f=\binom{a}{2}=\binom{m}{2}-\binom{b}{2}=c(m-c)$. Suppose that $r$ is the smallest integer such that $f$ can be written in the following form
\begin{equation}\label{equ:r}
\begin{split}
f=\sum_{i=1}^{r+1}\binom{x_i}{2},\mbox{ where integers } x_i\geq 1 \mbox{ satisfy } \sum_{i=1}^{r+1}x_i=m.
\end{split}
\end{equation}
Then $\sigma(m,f)\geq \frac1r.$ Moreover, for $r\geq 9$ it holds that $\sigma(m,f)=\frac1r.$
\end{thm}

Erd\H{o}s, F\"uredi, Rothschild and S\'os \cite{EFRS} further conjectured that the inequality in Theorem~\ref{Thm:Non zeros} should be an equality for any $r$.\footnote{See the paragraph before the proof of Theorem 3 in \cite{EFRS}.}
Note that for any integers $m,f$ from Theorem \ref{Thm:Non zeros},
we have $(m,f)\notin \mathcal{A}$ and thus by Theorem~\ref{Thm:EFRS}, any integer $r$ chosen from \eqref{equ:r} must satisfy $r\geq 2$.
We summarize this as the following.

\begin{conjecture}[\cite{EFRS}]\label{Conj:equzlity}
Let $m,f$ be integers from Theorem~\ref{Thm:Non zeros} and let $r\geq 2$ be from \eqref{equ:r}.
Then $\sigma(m,f)=\frac1r$.
\end{conjecture}

This remains open for $2\leq r\leq 8$.
It is worth pointing out that in addition to powerful results in extremal graph theory, the proof of each of the above two theorems in \cite{EFRS} used tools from number theory.

In this paper we provide improvements on the above two results of \cite{EFRS}.
We would like to emphasize that part of our proofs is based on elementary and analytic methods in number theory.
Our first result answers the question of \eqref{equ:1/2} in the affirmative.
In addition, we show that the constant $1/2$ cannot be lower and thus is sharp.

\begin{thm}\label{Thm:main}
Any pair $(m,f)\notin \mathcal{A}$ satisfies $\sigma(m,f)\leq \frac12$.
On the other hand, there are infinitely many pairs $(m,f)$ with $\sigma(m,f)=\frac12.$
\end{thm}

The proof of Theorem~\ref{Thm:main} uses bounds on $\sigma(m,f)$ established by Erd\H{o}s et. al in \cite{EFRS}.
As a corollary, Theorem~\ref{Thm:main}, together with the lower bound of Theorem~\ref{Thm:Non zeros},
implies the verification of Conjecture~\ref{Conj:equzlity} for the case $r=2$.
Our second result confirms more cases of Conjecture~\ref{Conj:equzlity}.

\begin{thm}\label{Thm:second thm}
Conjecture~\ref{Conj:equzlity} holds whenever $r=2$ or $r\geq 5$. 
\end{thm}

This is mainly built on a key concept introduced in \cite{EFRS} (see \eqref{equ:Cnr} in Section~\ref{sec:conj1.3}), which is related to number theory.
We will present two proofs of Theorem~\ref{Thm:second thm}, one for $r\geq 7$ using elementary arguments and another for $r\geq 5$ using analytic arguments.
We defer a more detailed discussion on Theorem~\ref{Thm:second thm} to Section~\ref{sec:conj1.3}.

The rest of the paper is organized as follows.
In Section~2, we collect some results from the literature.
In Section~3, we prove Theorem~\ref{Thm:main}.
In Section~4, we prove Theorem~\ref{Thm:second thm}.
In Section~5, we consider a conjecture of Erd\H{o}s et. al \cite{EFRS} on the existence of $\lim_{n\to \infty} |S_n(m,f)|/\binom{n}{2}$ and conclude with some remarks.

\section{Preliminaries}



We now prepare some results needed in later sections.
The following lemma on $\sigma(m,f)$ is collected from Sections 4 and 8 of \cite{EFRS} (see the equations (4.1), (4.4), (8.1), (8.3), (8.4) and (8.6) therein, respectively).

\begin{lem}[\cite{EFRS}]\label{lem:P}
\begin{itemize}
\item [(i).] $\sigma(m,f)=\sigma\big(m,\binom{m}{2}-f\big)$. 
\item [(ii).] If $\sigma(m,f)>\frac12$, then $\lfloor\frac{(m-1)^2}{4}\rfloor\leq f\leq \lfloor\frac{m^2}{4}\rfloor$.
\item [(iii).] Write $f=\binom{b}{2}-b'$ for integers $b,b'$ with $0\le b'< b-1$. If $\frac12 b<b'<b-1,$ then $\sigma(m,f)\le \frac12.$
\item [(iv).] Write $f=\binom{\ell}{2}+\ell'$ for integers $\ell, \ell'$ with $0\le \ell'<\ell<m$. If $\ell'\ge m-\ell$, then $\sigma(m,f)=0.$
\item [(v).] Let $D(m)$ denote the set of integers $xy+z$, where $x,y,z$ are nonnegative integers satisfying that $x+y\le m$ and if $z\ge 1$ then $x+y+z\le m-1$.
If $\sigma(m,f)>\frac12,$ then $f\in D(m)$.
\end{itemize}
\end{lem}

The following concentration inequality can be found in \cite{GIKM} (see its Corollary 2.2).

\begin{lem}[\cite{GIKM}]\label{Lem:Concentration}
Let $\binom{[N]}{n}$ be the set of $n$-subset of $\{1,2,...,N\}$ and let $h:\binom{[N]}{n}\rightarrow\mathbb{R}$ be a given function.
Let $C$ be a uniformly random element of $\binom{[N]}{n}.$
Suppose that there exists $\alpha>0$ such that
$|h(A)-h(A')|\le \alpha$ for any $A,A'\in\binom{[N]}{n}$ with $|A\cap A'|=n-1.$ Then for any real $t>0,$
$P(|h(C)-\mathbb{E}[h(C)]|\ge t)\le 2\exp\Big(-\frac{2t^2}{\min\{n,N-n\}\alpha^2}\Big).$
\end{lem}

We also need two lemmas from number theory. Let $\dN$ denote the set of non-negative integers.

\begin{lem}[Bennett, see \cite{EFRS}]\label{Lem:Bennett}
The equation $2\binom{x}{2}=\binom{y^2}{2}$ has a unique solution $(x,y)=(3,2)$ in positive integers.
\end{lem}

\begin{lem}[Gauss]\label{Lem:Gauss-Legendre}
	Define $\dG=\{x^2+y^2+z^2: ~\forall x,y,z\in\dN\}.$
	Then we have $\dN\backslash\dG=\{4^a(8b+7): a,b\in\dN\}.$
\end{lem}

\section{Proof of Theorem \ref{Thm:main}}
Recall that $\mathcal{A}= \{(2,0),(2,1),(4,3),(5,4),(5,6)\}$.
Our first goal is to show that any pair $(m,f)\notin \dA$ satisfies $\sigma(m,f)\leq \frac12$.
Suppose for a contradiction that there exists some pair $(m,f)\notin \dA$ with $\sigma(m,f)>\frac12$.

By Lemma~\ref{lem:P} (ii), we have
\begin{equation}\label{Equ:f}
\left\lfloor\frac{(m-1)^2}{4}\right\rfloor\le f\le \left\lfloor\frac{m^2}{4}\right\rfloor.
\end{equation}

We first assert that $m\geq 8$.
Considering $m=7$, by \eqref{Equ:f} it suffices to consider $9\leq f\leq 12$.
By Lemma~\ref{lem:P} (i), we see $\sigma(7,10)=\sigma(7,11)$ and $\sigma(7,9)=\sigma(7,12)$.
Using Lemma~\ref{lem:P} (iii) (with $f=7,b=6,b'=4$), it follows that $\sigma(7,10)=\sigma(7,11)\leq \frac12$;
on the other hand, Lemma~\ref{lem:P} (iv) (with $f=12,\ell=5,\ell'=2$) implies that $\sigma(7,9)=\sigma(7,12)=0$.
The cases for $m\leq 6$ can be similarly verified as follows.
By \eqref{Equ:f}, Lemma~\ref{lem:P} (i) and the fact $(m,f)\notin \mathcal{A}$, it suffices to consider:
$\sigma(3,2)$ for $m\leq 3$, $\sigma(4,2)=\sigma(4,4)$ for $m=4$, $\sigma(5,5)$ for $m=5$, and
$\sigma(6,6)=\sigma(6,9)$ and $\sigma(6,7)=\sigma(6,8)$ for $m=6$.
Then one can apply Lemma~\ref{lem:P} (iv) to show that each of these above pairs $(m,f)$ satisfies $\sigma(m,f)=0$.
So we have $m\geq 8$.

From now on, we express $m\in \{2k, 2k+1\}$ for some integer $k\geq 4$,
and write $f$ in the form of $f=\binom{\ell}{2}+\ell'$ for the unique integers $\ell, \ell'$ with $0\le \ell'<\ell$.
Since $f<\binom{m}{2}$, we have $\ell<m$.
By Lemma~\ref{lem:P} (iv), we can derive that $0\le \ell'\le m-\ell-1$.

We claim that $\ell'=0$ and thus $f=\binom{\ell}{2}$.
Suppose on the contrary that $\ell'\geq 1$.
We can also write $f$ in the form $f=\binom{\ell+1}{2}-(\ell-\ell')$ where $0<\ell-\ell'<\ell$.
By Lemma~\ref{lem:P} (iii),
if $\frac{\ell+1}{2}<\ell-\ell'<\ell$, then $\sigma(m,f)\leq \frac12$, a contradiction.
Thus, we have $\ell-\ell'\leq \frac{\ell+1}{2}$.
This together with $\ell'\le m-\ell-1$ implies that $\ell\leq \frac{2m-1}{3}$.
We discuss according to the parity of $m$.
First, let us consider when $m=2k$.
Then \eqref{Equ:f} implies $k^2-k\le f=\binom{\ell}{2}+\ell'\leq \binom{\ell}{2}+2k-\ell-1$, and solving this, we can obtain
$$\frac{3+\sqrt{8k^2-24k+17}}{2}\le \ell \le \frac{2m-1}{3}=\frac{4k-1}{3}.$$
Rearranging both sides, it gives $(k-1)(k-4)\leq 0$ and thus $1\leq k\leq 4$.
Note that $k\geq 4$.
So we have $k=4$ and further, we can derive that $m=8, \ell=5, \ell'=2$ and $f=12$.
Since $16=\binom{7}{2}-5$, by Lemma~\ref{lem:P} (i) and (iii), $\sigma(m,f)=\sigma(8,12)=\sigma(8,16)\le\frac12$, a contradiction.
Hence, we may assume that $m=2k+1$.
In this case, \eqref{Equ:f} infers that $k^2\le f=\binom{\ell}{2}+\ell'\leq \binom{\ell}{2}+2k-\ell$,
which implies that
$$\frac{3+\sqrt{8k^2-16k+9}}{2}\leq \ell\leq \frac{2m-1}{3}=\frac{4k+1}{3}.$$
The above inequality gives $k=2$, a contradiction to that $k\geq 4$. This proves the claim that $\ell'=0$.

By \eqref{Equ:f}, we have $k^2-k\leq f=\binom{\ell}{2}\leq k^2$ for $m=2k$ and $k^2\leq f=\binom{\ell}{2}\leq k^2+k$ for $m=2k+1$.
This leads to
\begin{align*}
\sqrt{2}k-1<\frac{1+\sqrt{8k^2-8k+1}}{2}\leq \ell\leq \frac{1+\sqrt{8k^2+1}}{2}<\sqrt{2}k+1 \mbox{ ~for } m=2k, \mbox{ and } \\
\sqrt{2}k<\frac{1+\sqrt{8k^2+1}}{2}\leq \ell\leq \frac{1+\sqrt{8k^2+8k+1}}{2}<\sqrt{2}k+2 \mbox{ ~for } m=2k+1.
\end{align*}
Since $\sqrt2 k\notin \mathbb{N}$ for $k\in \mathbb{N}\backslash \{0\}$,
we see that
\begin{equation}\label{Equ:ell}
\ell\in \{\lfloor \sqrt{2}k\rfloor, \lfloor \sqrt{2}k\rfloor+1\} \mbox{ for } m=2k
\mbox{ and } \ell\in \{\lfloor \sqrt{2}k\rfloor+1, \lfloor \sqrt{2}k\rfloor+2\} \mbox{ for } m=2k+1.
\end{equation}

Next we prove that $f=\binom{\ell}{2}=\frac12\binom{m}{2}$.
Suppose on the contrary that $f':=\binom{m}{2}-f \neq f$. Since $\sigma(m,f')=\sigma(m,f)>\frac12$,
repeating the above arguments, we can conclude that $f'=\binom{g}{2}$, where $g$ is a positive integer instead of $\ell$ satisfying \eqref{Equ:ell}.
Since $\lfloor\frac{(m-1)^2}{4}\rfloor\le f\neq f'\le \lfloor\frac{m^2}{4}\rfloor$,
it follows that $|f-f'|\leq \lfloor\frac{m^2}{4}\rfloor-\lfloor\frac{(m-1)^2}{4}\rfloor= k$.
As $g\neq \ell$ and both $g, \ell$ satisfy \eqref{Equ:ell},
we can derive that $$k\ge |f-f'|=\left|\binom{\ell}{2}-\binom{g}{2}\right|\ge\min\{\ell,g\}\ge \left\lfloor\sqrt2k\right\rfloor,$$
which is a contradiction for $k\geq 4$.
This proves that $2\binom{\ell}{2}=2f=\binom{m}{2}$, as desired.

Finally we show that $m$ must be a perfect square. By Lemma~\ref{lem:P} (v), we see $f\in D(m)$.
That is, $f$ can be written as $f=xy+z$ for nonnegative integers $x,y$ and $z$ satisfying that $x+y\le m$ and if $z\ge 1$ then $x+y+z\le m-1$.
If $z\ge1$, as $m\ge 8$,
we can infer that $f=xy+z\le \lfloor\frac{(m-1-z)^2}{4}\rfloor+z\le \lfloor\frac{(m-2)^2}{4}\rfloor+1<\lfloor\frac{(m-1)^2}{4}\rfloor$,
a contradiction to \eqref{Equ:f}.
Thus we have $z=0$ and $f=xy$ where $x+y\le m$.
If $x+y\le m-1$, then we have $f=xy\le\lfloor\frac{(m-1)^2}{4}\rfloor$ which contradicts that $f=\frac12\binom{m}{2}$.
So we must have $f=x(m-x)$ for some nonnegative integer $x$.
Solving $x(m-x)=f=\frac12\binom{m}{2}$, we get $x=\frac12(m\pm\sqrt m)$ which implies that $m$ is a perfect square.

Therefore, we have $2\binom{\ell}{2}=\binom{m}{2}$ where $m$ is a perfect square.
By Lemma \ref{Lem:Bennett}, this equation has the unique solution $(\ell,m)=(3,4)$ in positive integers,
which contradicts that $m\ge 8$.
This completes the proof of the first assertion of Theorem\ref{Thm:main} that any $(m,f)\notin \dA$ satisfies $\sigma(m,f)\leq \frac12$.

\medskip

To show the second assertion of Theorem \ref{Thm:main}, we construct an infinite sequence of pairs $(m,f)$ with $\sigma(m,f)=\frac12.$
In view of the first assertion, it is enough to show infinitely many pairs $(m,f)$ with $\sigma(m,f)\geq \frac12.$
Using Theorem \ref{Thm:Non zeros}, we can further reduce to find infinitely many $(m,f)$ satisfying the following properties:
\begin{itemize}
\item [(A).] $f$ can be expressed as $f=\binom{a}{2}=\binom{m}{2}-\binom{b}{2}=c(m-c)$ for some positive integers $a,b,c$,
\item [(B).] $f$ can be expressed as $f=\binom{x_1}{2}+\binom{x_2}{2}+\binom{x_3}{2}$ for integers $x_i\geq 1$ with $x_1+x_2+x_3=m$, and
\item [(C).] $f$ cannot be expressed as $f=\binom{y_1}{2}+\binom{y_2}{2}$ for integers $y_i\geq 1$ with $y_1+y_2=m$.
\end{itemize}
To do so, we will make use of the coming two equations, which can be easily verified for any integer $t\geq 1$:
$$\binom{5t+2}{2}=\binom{3t+1}{2}+\binom{4t+2}{2} \mbox{ ~and~ } \binom{3t+1}{2}=\binom{2t+1}{2}+\binom{2t+1}{2}+\binom{t}{2}.$$
We define $m=5t+2$ and $f=\binom{3t+1}{2}$.
By the above equations, we see that (B) automatically holds for such $(m,f)$, and one can choose integers $a=3t+1$ and $b=4t+2$ in (A).
Next, we show that such $(m,f)$ also satisfies (C);
as otherwise, by Jensen's inequality, we can derive that for any $t\geq 1$,
$$f=\binom{y_1}{2}+\binom{5t+2-y_1}{2}\ge\binom{\lfloor2.5t\rfloor+1}{2}+\binom{\lceil2.5t\rceil+1}{2}>\binom{3t+1}{2}=f,$$
a contradiction.
We are left to find a positive integer $c$ such that $\binom{3t+1}{2}=c(5t+2-c)$, which implies $c=\frac12(5t+2-\sqrt{7t^2+14t+4}).$
Note that for even integers $t\geq 1$, if $\sqrt{7t^2+14t+4}$ is an integer, then $c$ must be a positive integer.
Hence, it suffices to find infinitely many even integers $t\geq 1$ such that $\sqrt{7t^2+14t+4}$ is an integer.
By letting $x=\sqrt{7t^2+14t+4}$ and $y=t+1$,
our task is now to find infinitely many positive integer solutions $(x,y)$ to the Pell's equation
\begin{equation}\label{Equ:Pell}
x^2-7y^2=-3,
\end{equation}
where $x$ is even and $y$ is odd.
Note that $(x,y)=(2,1)$ is a positive solution to \eqref{Equ:Pell}.
We also observe that
$(x+y\sqrt7)(8+3\sqrt7)=(8x+21y)+(3x+8y)\sqrt7$ and if $(x,y)$ is a positive integer solution to \eqref{Equ:Pell}, then so is $(8x+21y,3x+8y)$.
Combining these facts together, we obtain infinite positive integer solutions $(x_k,y_k)$ to \eqref{Equ:Pell}, where $x_k+y_k\sqrt7=(2+\sqrt7)(8+3\sqrt7)^k$ for all $k\geq 0$.
Using the recurrences that $x_{k+1}=8x_k+21y_k$ and $y_{k+1}=3x_k+8y_k$,
it is easy to see that $x_{2k}$ is even and $y_{2k}$ is odd.

To give an explicit formula for the above construction, let $t_{k}=y_{2k}-1$ and we can derive that
$$t_{k}=\left(\begin{array}{c} 0 \\
1\end{array}\right)^T
\left(\begin{array}{cc} 8 & 21\\
3 & 8\end{array}\right)^{2k}
\left(\begin{array}{c} 2\\
1\end{array}\right)-1.$$
Now let $m_k=5t_{k}+2$ and $f_k=\binom{3t_{k}+1}{2}$.
By the above analysis, we have $\sigma(m_k,f_k)=\frac12$ for all $k\geq 1$.
This finishes the proof of Theorem \ref{Thm:main}.\QED

\section{Proof of Theorem~\ref{Thm:second thm}}\label{sec:conj1.3}
In this section, we complete the proof of Theorem~\ref{Thm:second thm}.
First, let us prove the case $r=2$. 

\medskip

\noindent {\bf Proof of Theorem~\ref{Thm:second thm} for $r=2$}.
Let $m,f$ be integers from Theorem~\ref{Thm:Non zeros} with $r=2$.
Then $(m,f)\notin \mathcal{A}$ and by Theorem \ref{Thm:main}, we have $\sigma(m,f)\le\frac12$.
By Theorem \ref{Thm:Non zeros}, we get $\sigma(m,f)\ge\frac12$. Thus $\sigma(m,f)=\frac12$.
 \QED

\bigskip

To prove other cases of Theorem \ref{Thm:second thm}, we introduce the following concept given in \cite{EFRS}.
Let
\begin{equation}\label{equ:Cnr}
C(n,r)=\left\{\sum\limits_{i=1}^{r}\binom{n_i}{2}:~ \sum\limits_{i=1}^{r}n_i=n \mbox{~~and~~}  n_i\in \dN  \mbox{~~for~~} 1\le i\le r\right\}.
\end{equation}
So $C(n,r)$ consists of all possible numbers of edges in an $n$-vertex graph formed by at most $r$ cliques.
A direct application of the Cauchy-Schwarz inequality shows that the minimum element in $C(n,r)$ is at least $n^2/2r-n/2$, thus implying $|C(n,r)|\leq  \frac{n^2}{2}-\frac{n^2}{2r}$.
The following result was given in \cite{EFRS} implicitly, which reveals the importance of $C(n,r)$ for Conjecture~\ref{Conj:equzlity},
that is, if $C(n,r)$ is almost full, then Conjecture~\ref{Conj:equzlity} holds for such $r$.

\begin{lem}[\cite{EFRS}]\label{lem:relation}
Let $r\geq 2$. If $|C(n,r)|=\frac{n^2}{2}-\frac{n^2}{2r}+o(n^2)$, then Conjecture~\ref{Conj:equzlity} holds for the case $r$.
\end{lem}
\begin{proof}
Let $m,f$ be integers from Theorem~\ref{Thm:Non zeros} and let $r\geq 2$ be from \eqref{equ:r}.
Suppose $|C(n,r)|=\frac{n^2}{2}-\frac{n^2}{2r}+o(n^2)$.
Take any $e\in C(n,r)$. Then there exists an $n$-vertex graph $G$ with $e$ edges formed by at most $r$ cliques.
Clearly, any $m$-vertex subgraph of $G$ is a subgraph consisting of at most $r$ cliques.
By the choice of $(m,f)$, $f\in C(m,r+1)\backslash C(m,r).$
This shows that any $m$-vertex induced subgraph of $G$ cannot have $f$ edges.
So $e\notin S_n(m,f)$.
That also says, $S_n(m,f)\cap C(n,r)=\emptyset$.
Therefore, we have
$\sigma(m,f)\le \limsup_{n\to \infty} \left(1-|C(n,r)|/\binom{n}{2}\right)= \limsup_{n\to \infty}\left(\frac{1}{r}+o(1)\right)=\frac1r,$
and thus the equality holds for the case $r$ of Conjecture~\ref{Conj:equzlity}.
\end{proof}

Brueggeman and Hildebrand (unpublished, see \cite{EFRS}) showed that there exists a constant $c_r>0$ such that
\begin{equation*}\label{Equ:B-H equ}
\left[\frac{n^2}{2r}+c_rn,\frac{n^2-n}{2}-c_rn^{3/2}\right]\subseteq C(n,r) \mbox{ ~for each~ }r\ge 9.
\end{equation*}
This, together with Lemma~\ref{lem:relation}, was applied by Erd\H{o}s et. al in \cite{EFRS} to derive the equality $\sigma(m,f)=\frac{1}{r}$ for $r\geq 9$ in Theorem \ref{Thm:Non zeros}.
By the above discussion, to complete the proof of Theorem \ref{Thm:second thm}, it suffices to prove the following.
\begin{thm}\label{Thm:Key Thm}
	Let $r\ge 5$. 
	Then we have $|C(n,r)|=\frac{n^2}{2}-\frac{n^2}{2r}+o(n^2).$
\end{thm}

We remark (as we shall see later in the proof of Theorem~\ref{Thm:r>6}) that if Theorem~\ref{Thm:Key Thm} holds for the case $s$, then it holds for any $r\geq s$.

\subsection{An elementary proof for $r\geq 7$}

We first prove the following weak version of Theorem~\ref{Thm:Key Thm} by elementary arguments.

\begin{thm}\label{Thm:r>6}
Let $r\ge 7.$ Then for some constant $c_r>0$, we have
$\left[\frac{n^2}{2r}+c_rn, \frac{n^2-n}{2}-c_rn^{3/2}\right]\subseteq C(n,r).$
\end{thm}

\noindent{\it Proof.}
We first show that it suffices to handle the case $r=7.$
Suppose that $\Big[\frac{n^2}{2r}+c_rn, \frac{n^2-n}{2}-c_rn^{3/2}\Big]\subseteq C(n,r)$ holds.
We claim that this will lead to the analog statement for the case $r+1.$
For any $n_0\in C\Big(n-\lfloor\frac{n}{r+1}\rfloor,r\Big),$
we have $n_0+\binom{\lfloor\frac{n}{r+1}\rfloor}{2}\in C(n,r+1),$ that is, $C\Big(n-\lfloor\frac{n}{r+1}\rfloor,r\Big)+\binom{\lfloor\frac{n}{r+1}\rfloor}{2}\subseteq C(n,r+1).$
Also because $C(n,r)\subseteq C(n,r+1)$,
a careful calculation would give $\Big[\frac{n^2}{2(r+1)}+c_{r+1}n, \frac{n^2-n}{2}-c_{r+1}n^{3/2}\Big]\subseteq C(n,r+1)$ for some $c_{r+1}>0$, as claimed.

	For $r=7,$ we will prove that
	$$\left[\frac{n^2}{14}+\frac{n}{2}+2100, \frac{n^2-n}{2}-66n^{3/2}\right]\subseteq C(n,7).$$
	
	\noindent Let $m\in\Big[\frac{n^2}{14}+\frac{n}{2}+2100, \frac{n^2-n}{2}-66n^{3/2}\Big]$ be an integer.
	Note that $m\in C(n,7)$ if and only if there exist non-negative integers $n_1,n_2,\ldots,n_7$ such that
	$\sum_{i=1}^{7}n_i^2=2m+n\ \text{and}\ \sum_{i=1}^{7}n_i=n.$	
	Therefore, in order to prove $m\in C(n,7)$, we only need to find integer solutions to
	\begin{align}\label{Equ:eq 1}
	\sum_{i=1}^{3}\big(x_i^2+(2t-x_i)^2\big)+(n-6t)^2=2m+n,
	\end{align}
	where $0\le x_i\le 2t\ (1\le i\le 3)$ and $6t\le n.$ The equation \eqref{Equ:eq 1} is equivalent to
	\begin{align*}
	2\sum_{i=1}^{3}(x_i-t)^2=2m+n-(n-6t)^2-6t^2.
	\end{align*}
	In view of Lemma \ref{Lem:Gauss-Legendre},
	the equation \eqref{Equ:eq 1} is solvable if there exists $t\in \dN$ such that
	\begin{align}\label{Equ:eq 2}
	6t\le n,\ 0\le 2m+n-(n-6t)^2-6t^2\le t^2\ \text{and}\  2m+n-(n-6t)^2-6t^2\in2\dG,
	\end{align}
	where $\dG$ is from Lemma \ref{Lem:Gauss-Legendre} and the second inequality insure that $0\le x_i\le 2t$ for $1\le i\le 3$.

	We denote $f(t):=f(t;m,n)=2m+n-(n-6t)^2-6t^2=-42(t-\frac{n}{7})^2-\frac{n}{7}^2+n+2m.$
	When $0\le t\le \frac{n}{7},$ $f(t)$
	is an increasing function with $f(0)=-n^2+n+2m$ and $f(\frac{n}{7})=-\frac{n}{7}^2+n+2m.$
	Since $\frac{n^2}{14}+\frac{n}{2}+2100\le m\le \frac{n^2-n}{2}-66n^{3/2},$
	we have $f(0)<0$ and $f(\frac{n}{7})>f(\frac{n}{7}-10)=-4200-\frac{n}{7}^2+n+2m\ge 0.$
	Now we can choose an integer $t_0$ with $0\le t_0\le \frac{n}{7}-10$ such that $f(t_0)\le 0$ and
	$f(t_0+1)>0.$
	Next we show there exists an integer $t_0+1\le t\le t_0+10,$ satisfying \eqref{Equ:eq 2}.
	
	We first claim that $t_0>11\sqrt{n}-1.$
	Indeed, suppose that $t_0\le 11\sqrt{n}-1,$
	then we can see that
	$0<f(t_0+1)\le f(11\sqrt{n})=-5082n+132n^{3/2}-n^2+n+2m\le -5082n<0,$ which is a contradiction.
	It is easy to see that $f(t+1)-f(t)=-84t-42+12n,$ and if $t>11\sqrt{n}-1,$
	then
	$f(t+10)-f(t)=-840t-4200+120n<t^2.$
	Thus, for any $t_0+1\le t\le t_0+10\le n/7$,
	we have $6t\leq 6n/7<n$ and $0<f(t_0+1)\le f(t)\le f(t_0+10)\le f(t_0)+t_0^2\le t_0^2<t^2.$
	It is left to explain that one can choose $t$ with  $t_0+1\le t\le t_0+10$
	such that $f(t)\in 2\dG.$
	
	Note that $f(t)$ is alwalys an even integer.
	Also note that $8b+7\equiv 7,15\,(mod\ 16),$
	$4(8b+7)\equiv 12\,(mod\ 16)$
	and $4^a(8b+7)\equiv 0\,(mod\ 16)$ for all $a\ge 2.$
	Thus by Lemma \ref{Lem:Gauss-Legendre},
	it suffices to find $f(t)$ such that $\frac{f(t)}{2}\not\equiv 0,7,12,15\,(mod\ 16).$
	It is clear that we can choose $t'$ with $t_0+1\le t'\le t_0+8,$ such that
	$t'+n\equiv 0\ (mod\ 8)$. Then
	$$\frac{f(t'+1)-f(t')}{2}=-42t'-21+6n\equiv 11\,(mod\ 16) \mbox{ ~~~and~~~ }$$
	$$\frac{f(t'+2)-f(t'+1)}{2}=-42(t'+1)-21+6n\equiv 1\,(mod\ 16).$$
	We can easily see from above that there exists $t\in\{t',t'+1,t'+2\}$ such that
	$\frac{f(t)}{2}\not\equiv 0,7,12,15\,(mod\ 16).$
	This completes the proof of Theorem \ref{Thm:r>6}. 	\QED

\subsection{Proof of Theorem \ref{Thm:Key Thm}}\label{subsec:4.2}
In this subsection, we will use 
analytic method to prove Theorem \ref{Thm:Key Thm}.
Let $f$ and $g$ be two functions on the same domain, where $f$ takes complex values and $g$ takes non-negative real values.
We use the Vinogradov symbols $f \ll g$, if there exists a constant $C>0$ such that $|f|\le Cg$.\footnote{Here and in the rest, $|f|$ denotes the modulus of a complex number $f$.}

As the previous analysis in the proof of Theorem \ref{Thm:r>6},
we may assume that $r=5$ in the following proof.
Indeed, since $C(n,r)\subseteq C(n,r+1)$ and
$C\Big(n-\lfloor\frac{n}{r+1}\rfloor,r\Big)+\binom{\lfloor\frac{n}{r+1}\rfloor}{2}\subseteq C(n,r+1),$
then by the definition of $C(n,r)$ and induction on $r,$ we can get the desired result for $r+1$.

Let
$\dE(n)=\Big\{\frac{n^2}{10}+\frac{n^2}{\log n}\le m\le \frac{n^2-n}{2}-\frac{n^2}{\log n}:\ \dR(m)=0\Big\},$
with
\begin{equation}\label{Equ:definition of R(m)}
	\dR(m)=
	\sum_{\substack{
			1\le x_1,x_2,x_3,x_4\le \frac{n}{5}-\frac{n}{\log n}\\
			x_1^2+x_2^2+x_3^2+x_4^2+(x_1+x_2+x_3+x_4-n)^2=2m+n}}
	1,
\end{equation}
where $x_1,x_2,x_3,x_4$ are integers.
Note that $\dR(m)>0$ implies that $m\in C(n,5).$
Therefore, to prove Theorem \ref{Thm:Key Thm}, it is enough to prove $|\dE(n)|=o(n^2).$
We will show the following stronger result.
Let $\mathbb{Z}$ be the set of integers.

\begin{thm}\label{Thm:E(n)}
	$|\dE(n)|=O\big(n^{2-\frac{1}{50}}\big).$
\end{thm}

We denote $N=\frac{n}{5}-\frac{n}{\log n}.$
For positive integers $x_1,x_2,x_3,x_4,$
we use $\vec{\textbf x}$=$\{x_1,x_2,x_3,x_4\}$ to denote vectors in $\mathbb{Z}^{4}.$
Let $Q(\vec{\textbf x})=Q(x_1,x_2,x_3,x_4)=x_1^2+x_2^2+x_3^2+x_{4}^2+(x_1+x_2+x_3+x_{4}-n)^2.$
For a vector $\vec{\textbf y}$=$\{y_1,y_2,y_3,y_{4}\}\in \mathbb{Z}^{4},$
the notation $\vec{\textbf x}\le \vec{\textbf y}$
means $x_i\le y_i$ for all $1\le i\le 4,$ and
$1\le\vec{\textbf x}\le C$
means $1\le x_i\le C$ for all $1\le i\le 4.$
For convenience, we write $\e(\alpha)=e^{2\pi i \alpha}.$
Define
\begin{equation}\label{Equ:definition of f(a)}
f(\alpha)=\sum\limits_{1\le \vec{\textbf x} \le N}\e\big(\alpha Q(\vec{\textbf x})\big).
\end{equation}
Then by \eqref{Equ:definition of R(m)} and \eqref{Equ:definition of f(a)}, we have
\begin{equation}\label{Equ:R(m) equality}
\dR(m)=\int_{\frac1n}^{1+\frac1n}f(\alpha)\e\big(-\alpha (2m+n)\big)\, d\alpha.
\end{equation}
We define
\begin{equation*}
\dM(X)=
\bigcup_{1\le q\le X}
\bigcup_{\substack{1\le a\le q\\
		(a,q)=1}}
\left[\frac{a}{q}-\frac{X}{qn^2},\frac{a}{q}+\frac{X}{qn^2}\right]\ \ \text{and}\ \  \dm(X)=\left[\frac1n,1+\frac1n\right]\backslash\dM(X).
\end{equation*}

Note that the above union is pairwise disjoint for $X\le \frac{n}{2}.$
Also note that both $\int_{\dM(L)}f(\alpha)\e\big(-\alpha(2m+n)\big)\,d\alpha$ and
$\int_{\dm(L)}f(\alpha)\e\big(-\alpha(2m+n)\big)\,d\alpha$ take real values.
We first estimate the integral on $\dm(X)$ in the following lemma.

\begin{lemma}\label{Lem:m estimation}
	Let $L\le \frac{n}{2}.$
	Then
	$\sum_{\frac{n^2}{10}+\frac{n^2}{\log n}\le m\le \frac{n^2-n}{2}-\frac{n^2}{\log n}}\big|\int_{\dm(L)}f(\alpha)\e\big(-\alpha (2m+n)\big)\, d\alpha\big|^2=O\left(\frac{n^{6}(\log n)^5}{L^{2}}\right).$
\end{lemma}

To show this lemma, we need the following claim.

\begin{claim}\label{Clm:estimation of f(a)^2}
	Suppose that $|\alpha-\frac{a}{q}|\le\frac{1}{q^2}$ with $(a,q)=1.$
	Then
	$$|f(\alpha)|^2\ll n^{8}(\log n)^{4}\left(\frac{1}{q}+\frac{1}{n}+\frac{q}{n^2}\right)^{4}.$$
\end{claim}
\begin{proof}
	This can be proved by the standard difference argument.
	Note that
	\begin{align*}
	|f(\alpha)|^2&=\sum\limits_{1\le\vec{\textbf x},\ \vec{\textbf y}\le N} \e\Big(\alpha\Big(Q(\vec{\textbf x})-Q(\vec{\textbf y})\Big)\Big)
  	=\sum\limits_{-N+1\le\vec{\textbf h}\le N-1}
   	  \sum_{\substack{
   			1\le \vec{\textbf y}\le N\\
   			1\le \vec{\textbf y}+\vec{\textbf h}\le N}}
   	\e\Big(\alpha\Big(Q(\vec{\textbf y}+\vec{\textbf h})-Q(\vec{\textbf y})\Big)\Big).
	\end{align*}
	By triangle inequality, we deduce that
	$$|f(\alpha)|^2\le
	\sum\limits_{-N+1\le\vec{\textbf h}\le N-1}
	\Big|\sum_{\substack{
			1\le \vec{\textbf y}\le N\\
			1\le \vec{\textbf y}+\vec{\textbf h}\le N}}
	\e\Big(\alpha\Big(Q(\vec{\textbf y}+\vec{\textbf h})-Q(\vec{\textbf y})\Big)\Big)\Big|.
	$$
	Fix $\vec{\textbf h}$ and let $h=\sum_{i=1}^{4}h_i.$ Note that
	\begin{align*}
		\Big|\sum_{\substack{
				1\le \vec{\textbf y}\le N\\
				1\le \vec{\textbf y}+\vec{\textbf h}\le N}}
		\e\Big(\alpha\Big(Q(\vec{\textbf y}+\vec{\textbf h})-Q(\vec{\textbf y})\Big)\Big)\Big|
		\le	\Big|\sum_{\substack{
					1\le \vec{\textbf y}\le N\\
					1\le \vec{\textbf y}+\vec{\textbf h}\le N}}
			\e\Big(2\alpha \sum_{j=1}^{4}(h+h_j)y_j\Big)\Big|
		\le\prod_{j=1}^{4}	
		\Big|\sum_{\substack{
		1\le y_j\le N\\
		1-h_j\le y_j\le N-h_j}}
        \e\Big(2\alpha (h+h_j)y_j\Big)\Big|.
	\end{align*}
	Write $\|\beta\|=\min_{m\in \mathbb{Z}}|\beta-m|$. Since
	$\sum_{\substack{
			1\le y_j\le N\\
			1-h_j\le y_j\le N-h_j}}
	\e\Big(2\alpha (h+h_j)y_j\Big)
	\ll
	\min\big(N, \lVert 2\alpha (h+h_j)\rVert^{-1}\big),$
	we obtain
	\begin{align*}
	|f(\alpha)|^2
	\ll
	\sum\limits_{-N+1\le\vec{\textbf h}\le N-1}\prod_{j=1}^{4}
	\min\big(N, \lVert 2\alpha (h+h_j)\rVert^{-1}\big)
	\ll
	\sum\limits_{-10(N-1)\le\vec{\textbf t}\le 10(N-1)}w(\vec{\textbf t})\prod_{j=1}^{4}
	\min\big(N, \lVert\alpha t_j\rVert^{-1}\big),
	\end{align*}
	where
	$w(\vec{\textbf t})=\sum\limits_{-N+1\le\vec{\textbf h}\le N-1, ~2(h+h_j)=t_j, ~\forall 1\le j\le 4} 1.$
	It is easy to see that $w(\vec{\textbf t})\in\{0,1\},$ and thus we have
	\begin{align}\label{Equ:min inequality}
		|f(\alpha)|^2
		\ll
		\sum\limits_{-10(N-1)\le\vec{\textbf t}\le 10(N-1)}\prod_{j=1}^{4}
		\min\big(N, \lVert\alpha t_j\rVert^{-1}\big)
		=\left(\sum\limits_{-10(N-1)\le t\le 10(N-1)}\min\big(N, \lVert\alpha t\rVert^{-1}\big)\right)^{4}
	\end{align}
	As $|\alpha-\frac{a}{q}|\le\frac{1}{q^2}$ and $(a,q)=1$,
	by Lemma 2.2 in \cite{V09} (with $X=10N, Y=N/10$), we have
	\begin{align}\label{Equ:N^2logN}
	\sum\limits_{-10(N-1)\le t\le 10(N-1)}\min\big(N, \lVert\alpha t\rVert^{-1}\big)
	\ll
	N^2(\log N)\left(\frac1q+\frac1N+\frac{q}{N^2}\right).
	\end{align}
	This, together with \eqref{Equ:min inequality}, \eqref{Equ:N^2logN} and the definition of $N$, completes the proof of Claim \ref{Clm:estimation of f(a)^2}.	
\end{proof}



\medskip
\noindent\textbf{Proof of Lemma \ref{Lem:m estimation}.}
    We can see that
    $\int_{\dm(L)}f(\alpha)\e\big(-\alpha (2m+n)\big)\, d\alpha$
    is the Fourier coefficient of the function which is $f(\alpha)$ on $\dm(L)$ and $0$ otherwise.
	Hence, by Bessel's inequality, we have
	$$\sum\limits_{\frac{n^2}{10}+\frac{n^2}{\log n}\le m\le \frac{n^2-n}{2}-\frac{n^2}{\log n}}\Big|\int_{\dm(L)}f(\alpha)\e\big(-\alpha (2m+n)\big)\, d\alpha\Big|^2
	\le\int_{\dm(L)}|f(\alpha)|^2\,d\alpha.$$	
	For $X< \frac{n}{2},$ we define $\dn(X)=\dM(2X)\backslash \dM(X),$
	and for $X=\frac{n}{2},$ we write $\dn(X)=\Big[\frac1n,1+\frac1n\Big]\backslash\dM(X).$
	Let $t=\lceil \log_2\frac{n}{2L}\rceil.$
	By the dyadic argument and the definition of $\dn(X)$,
	we can see that
	$\dm(L)\subseteq \dn(L)\cup\dn(2L)\cup\cdots\cup\dn(2^{t-1}L)\cup\dn(\frac n2).$
	Thus we only need to prove that for $L\le X\le \frac{n}{2},$
	$\int_{\dn(X)}|f(\alpha)|^2\,d\alpha
	\ll\frac{n^{6}(\log n)^{4}}{X^{2}}.$
	
	By Dirichlet's approximation theorem (see Lemma 2.1 in \cite{V09}), for $\alpha\in \dn(X),$
	there exist $a, q\in\mathbb{N}$
	such that
	$|\alpha-\frac{a}{q}|\le\frac{X}{qn^2},$
	$1\le a\le q\le \frac{n^2}{X}$
	and $(a,q)=1.$
	Since $\alpha\notin\dM(X),$
	we further have $q>X.$
	Now it follows from Claim \ref{Clm:estimation of f(a)^2} that
	$\sup_{\alpha\in\dn(X)}|f(\alpha)|^2
	\ll\frac{n^{8}(\log n)^{4}}{X^{4}}.$
	Note that the measure of $\dn(X)$ is $|\dn(X)|\ll\frac{X^2}{n^2},$
	so we obtain
	$\int_{\dn(X)}|f(\alpha)|^2\,d\alpha\ll\frac{n^{6}(\log n)^{4}}{X^{2}},$
	which completes the proof of Lemma \ref{Lem:m estimation}.	
\QED

\bigskip

In order to estimate the contribution from $\dM(X)$ in \eqref{Equ:R(m) equality}, we define
\begin{equation}\label{Equ:definitions of S and T}
    S(q,a)=\sum_{1\le \vec{\textbf x}\le q}\e\left(\frac{a}{q}Q(\vec{\textbf x}) \right)
    \ \ \text{and}\ \
	T(q;m)=\frac{1}{q^4}
	\sum_{\substack{1\le a\le q\\
			(a,q)=1}}
	S(q,a)\e\left(-\frac{a}{q}(2m+n) \right).
\end{equation}

The following claim can be proved by some standard argument (see Lemmas 2.10 and 2.11 in \cite{V09}).

\begin{claim}\label{Clm:fuction q}
$T(q;m)$ is multiplicative as a function of $q$.
\end{claim}

Also $T(q;m)$ takes real values. In the next claim we bound $|T(q;m)|$ from above for prime powers $q$.

\begin{claim}\label{Clm:general T(q,m)}
	Assume that $p$ is a prime and $k\in\mathbb{Z}^+.$
	Then
	$|T(p^k;m)|\le c_p\cdot p^{-k}\Big(1-\frac 1p\Big),$
	where $c_2=4,$ $c_5=\sqrt 5$ and $c_p=1$ if $p\nmid10.$
\end{claim}
\begin{proof}
	We can deduce that
		\begin{align*}
		|S(q,a)|^2
		=\sum\limits_{1\le\vec{\textbf x},\ \vec{\textbf h}\le q} \e\left(\frac aq\Big(Q(\vec{\textbf x}+\vec{\textbf h})-Q(\vec{\textbf x})\Big)\right)
		\le\sum\limits_{1\le\vec{\textbf h}\le q}
		   \Big|\sum\limits_{1\le\vec{\textbf x}\le q} \e\Big(\frac aq\Big(2\sum_{j=1}^{4}(h_1+\cdots+h_4+h_j)x_j\Big)\Big)\Big|
		   \le q^4S_q,
		\end{align*}
	where $S_q$ is the number of solutions to
	$2(h_1+\cdots+h_4+h_j)\equiv 0\ (\text{mod}\ q)\ \ (1\le j\le 4)$
	with $1\le h_1,...,h_4\le q.$
	The last inequality holds by the fact that
	$\sum_{j=1}^q \e\Big(\frac aq t j\Big)=q,$
	when $t\equiv 0\ (\text{mod}\ q)$ and
	$\sum_{j=1}^q \e\Big(\frac aq t j\Big)=0$ otherwise.	
	Solving the congruence equations,
	we can get that
	$S_{2^k}=16,$ $S_{5^k}=5$ and $S_{p^k}=1$ if  $p\nmid10,$
	which implies that
	$|S(2^k,a)|\le 4\cdot 2^{2k},$ $|S(5^k,a)|\le \sqrt 5\cdot 5^{2k},$ and $|S(p^k,a)|\le p^{2k}$  if  $p\nmid10.$
Note that $|\{1\leq a\leq p^k:\ (a,p^k)=1\}|=p^k\Big(1-\frac 1p\Big)$. This completes the proof of Claim \ref{Clm:general T(q,m)} by \eqref{Equ:definitions of S and T}.
\end{proof}

In case $(q,10)=1,$
we can get the following better bound for $|T(q;m)|.$

\begin{claim}\label{Clm:T(q,m),(q,10)=1}
	Assume that $(q,10)=1$.
	Then $|T(q;m)|\le\frac{1}{q^2}\Big(q,n^2-5(2m+n)\Big).$
\end{claim}
\begin{proof}
	By Claim \ref{Clm:fuction q}, we only need to prove the claim for $q$ as a power of some prime $p.$
	So we may assume that $q=p^k,$ where $p$ is a prime with $(p,10)=1$ and $k\ge 1$ is an integer.
	Let $\bar{r}$ denotes an integer $r'$ satisfying $rr'\equiv 1\ (\text{mod}\ q).$
	Let
	$
	\textbf{A}=I_{4\times 4}+J_{4\times 4},
	$
	where $I$ is the identity matrix, and all the entries in $J$ are $1.$
	Note that $\det(\textbf{A})=5$ and $Q(\vec{\textbf{x}})=\vec{\textbf{x}}^T \textbf{A}\vec{\textbf{x}}-2n(x_1+x_2+x_3+x_4)+n^2.$
	Let $\vec{\textbf{b}}=\bar{5}n\textbf{A}^*\vec{\textbf{1}},$ where $\textbf{A}^*$ is the adjugate matrix of $\textbf{A}.$
    We get that
	$Q(\vec{\textbf{y}}+\vec{\textbf{b}})=\vec{\textbf{y}}^T \textbf{A}\vec{\textbf{y}}+\bar{5}n^2,$
	and
	$S(q,a)
	=\sum\limits_{1\le \vec{\textbf{y}}\le q}
	\e\Big(\frac aq\vec{\textbf{y}}^T \textbf{A}\vec{\textbf{y}}\Big)
	\e\Big(\frac aq\bar{5}n^2\Big)
	=\sum\limits_{1\le \vec{\textbf{y}}\le p^k}
	\e\Big(\frac{a}{p^k}\vec{\textbf{y}}^T \textbf{A}\vec{\textbf{y}}\Big)
	\e\Big(\frac{a}{p^k}\bar{5}n^2\Big).
	$
	
	Let $\vec{\textbf{y}}=\vec{\textbf{u}}p^{k-1}+\vec{\textbf{v}}$.
	When $k\ge 2,$ we have that
	$$\sum\limits_{1\le \vec{\textbf{y}}\le p^k}
	\e\left(\frac{a}{p^k}\vec{\textbf{y}}^T \textbf{A}\vec{\textbf{y}}\right)=
	\sum\limits_{1\le \vec{\textbf{u}}\le p}
	\sum\limits_{1\le \vec{\textbf{v}}\le p^{k-1}}
	\e\left(\frac{a}{p^k}\big(2\vec{\textbf{u}}^Tp^{k-1} \textbf{A}\vec{\textbf{v}}+\vec{\textbf{v}}^T \textbf{A}\vec{\textbf{v}}\big)\right)=
	p^2\sum\limits_{
		\substack{1\le \vec{\textbf{v}}\le p^{k-1}\\
			2\textbf{A}\vec{\textbf{v}}\equiv \vec{\textbf{0}}\ (\text{mod}\ p)}}
	\e\left(\frac{a}{p^k}\vec{\textbf{v}}^T \textbf{A}\vec{\textbf{v}}\right).
	$$
	The last equality holds since
	$\sum_{1\le \vec{\textbf{u}}\le p} \e\Big(\frac ap 2\vec{\textbf{u}}^T \textbf{A}\vec{\textbf{v}}\Big)=p^2,$
	when $2\textbf{A}\vec{\textbf{v}}\equiv \vec{\textbf{0}}\ (\text{mod}\ p)$ and
	$\sum_{1\le \vec{\textbf{u}}\le p} \e\Big(\frac ap 2\vec{\textbf{u}}^T \textbf{A}\vec{\textbf{v}}\Big)=0,$ otherwise.
	For $(p,10)=1,$
	$2\textbf{A}\vec{\textbf{v}}\equiv \vec{\textbf{0}}\ (\text{mod}\ p)$
	if and only if
	$\vec{\textbf{v}}\equiv \vec{\textbf{0}}\ (\text{mod}\ p).$
	Thus we have
	$\sum_{1\le \vec{\textbf{y}}\le p^k}
	\e\Big(\frac{a}{p^k}\vec{\textbf{y}}^T \textbf{A}\vec{\textbf{y}}\Big)
	=p^2\sum_{1\le \vec{\textbf{y}}\le p^{k-2}}
	\e\Big(\frac{a}{p^{k-2}}\vec{\textbf{y}}^T \textbf{A}\vec{\textbf{y}}\Big),$ for $k\ge 2.$
	By Lemma 26 in \cite{H-B} and its proof, we have
	$\sum_{1\le \vec{\textbf{y}}\le p}
		\e\Big(\frac{a}{p}\vec{\textbf{y}}^T \textbf{A}\vec{\textbf{y}}\Big)=\left(\frac{5}{p}\right)p^2,$
	where $\left(\frac5p\right)$ denotes the Jacobi symbol.
	Thus we get that
	$\sum_{1\le \vec{\textbf{y}}\le p^k}
	\e\Big(\frac{a}{p^k}\vec{\textbf{y}}^T \textbf{A}\vec{\textbf{y}}\Big)=\left(\frac{5}{p^k}\right)p^{2k}$ for all $k\ge 1,$
	and
	$S(q,a)=\left(\frac 5q\right)q^2\e\Big(\frac aq\bar{5}n^2\Big).$
	Then we can see that
	$|T(q;m)|
	=\frac{1}{q^2}
	\Big|{\sum_{\substack{1\le a\le q\\
			(a,q)=1}}}
		\e\Big(\frac aq\big(\bar{5}n^2-(2m+n)\big)\Big)
	\Big|
	\le \frac{1}{q^2}\Big(q,n^2-5(2m+n)\Big).$
	The last inequality is a standard upper bound of the Ramanujan sum.
\end{proof}

Define $G(m)=\sum_{q=1}^{\infty}T(q;m).$
By Claims \ref{Clm:fuction q}, \ref{Clm:general T(q,m)} and \ref{Clm:T(q,m),(q,10)=1},
$G(m)$ is absolutely convergent.

\begin{claim}\label{Clm:G(m)}
	Suppose that $n^2-5(2m+n)\neq0$.
	Then there is an absolute constant $c>0$ such that
	$G(m)\ge\frac{c}{\log\log n}.$
\end{claim}
\begin{proof}
	Let $\Delta=n^2-5(2m+n),$
	$P_1=\left\{\text{prime}\ p: p\nmid 10\Delta\right\},$
	$P_2=\left\{\text{prime}\ p: p\mid \Delta\ \text{and}\ p\nmid 10\right\}.$
	By Claim \ref{Clm:fuction q}, we have
	$G(m)=\prod\limits_{p\ prime}\Big(1+\sum_{k=1}^{\infty}T(p^k;m)\Big).$
	Combining Claims \ref{Clm:general T(q,m)} and \ref{Clm:T(q,m),(q,10)=1}, we get that
	\begin{equation*}
		1+\sum_{k=1}^{\infty}T(p^k;m)\ge
		\begin{cases}
		1-\sum_{k=1}^{\infty}\frac{1}{p^{2k}}=1-\frac{1}{p^2-1}\ge \Big(1-\frac{1}{p^2}\Big)^2&  \text{if}\ p\in P_1,
		\\
		1-\sum_{k=1}^{\infty}\frac{1}{p^{k}}=1-\frac{1}{p-1}&  \text{if}\ p\in P_2,
		\\
		1-\sum_{k=1}^{\infty}p^{-k+\frac12}\Big(1-\frac1p\Big)=1-p^{-\frac12}>0&  \text{if}\ p=5,
		\\
		1+T(p;m)-\sum_{k=2}^{\infty}p^{-k+2}\Big(1-\frac1p\Big)=T(2;m)=1& \text{if}\ p=2.
		\end{cases}
	\end{equation*}
	We first see that
    $\prod_{p\in P_1}\Big(1-\frac{1}{p^2}\Big)^2\ge \prod_{p}\Big(1-\frac{1}{p^2}\Big)^2=\zeta(2)^{-2}$.
	Then there exist constants $c_1,c_2>0$ such that
	$$G(m)\ge c_1\cdot
	\prod_{p\in P_2}\left(1-\frac{1}{p-1}\right)
	=c_1e^{\sum_{p\in P_2}\log \big(1-\frac{1}{p-1}\big)}
	\ge c_1e^{-\sum_{p\in P_2}\frac{1}{p-1}}\ge c_2e^{-\sum_{p| \Delta}\frac{1}{p}}.
	$$
	Let $Y=\log\Delta.$
	From Theorem 6.16 \cite{L96}, we have $\sum_{\substack{
					p\le Y
				}}\frac{1}{p}\le
	   	c_3\log\log Y$ for some constant $c_3>1$ and thus there exists constant $c_4>0$ such that
	$$\sum_{\substack{p\ prime\\
				p\mid \Delta
			}}\frac{1}{p}=
	   \sum_{\substack{p|\Delta\\
					p\le Y
				}}\frac{1}{p}+
	   \sum_{\substack{p |\Delta\\
	   		p> Y
	   	}}\frac{1}{p}\le\sum_{\substack{
					p\le Y
				}}\frac{1}{p}+
	   	\frac 1Y \sum_{\substack{p|\Delta\\
	   			p> Y
	   		}}1\le
	   	c_3\log\log Y+	\frac 1Y\cdot \frac{\log\Delta}{\log Y}
	    \le c_4 \log\log\log\Delta,
	$$
	which implies that $G(m)\ge c/\log\log\Delta$
	and completes the proof of Claim \ref{Clm:G(m)}.
\end{proof}

Define $I(\beta)=\int_{{[0,N]^4}}\e\Big(\beta Q(\vec{\textbf{y}})\Big)\, d\vec{\textbf{y}}$ and
$\mathfrak{I}(m)=\int_{-\infty}^{+\infty}I(\beta)\e\big(-\beta(2m+n)\big)\, d\beta.$
Note that $\mathfrak{I}(m)$ takes real values.
Using integration by parts, we have
$I(\beta)\ll |\beta|^{-2}.$
By Fourier inverse formula (see \cite{IK04}, p460), we have
$n^2(\log n)^{-10}\ll \mathfrak{I}(m)\ll n^2.$
For the contribution from $\dM(X)$,
we have the following.

\begin{lemma}\label{Lem:M estimation}
	Let $L=n^{\frac{1}{50}}.$ Then
	$\int_{\dM(L)}f(\alpha)\e\big(-\alpha(2m+n)\big)\,d\alpha=G(m)\mathfrak{I}(m)+O(n^{2-\frac{1}{100}}).$
\end{lemma}
\begin{proof}
Let $f^*(\alpha)=\frac{1}{q^4}S(q,a)I(\alpha-\frac aq).$
Suppose that $|\alpha-\frac aq|\le\frac{L}{qn^2}$
with $1\le a\le q\le L$ and $(a,q)=1,$
then the partial summation formular
(see Lemma 2.6 in \cite{V09} with $c_i=1,$ $X=N$)
implies
$f(\alpha)=f^*(\alpha)+O(N^3L).$
Thus we get that
$$
\int_{\dM(L)}f(\alpha)\e\big(-\alpha(2m+n)\big)\,d\alpha=
\int_{\dM(L)}f^*(\alpha)\e\big(-\alpha(2m+n)\big)\,d\alpha+O\big(N^3L\cdot |\dM(L)|\big)=R^*+O(N^3L^3n^{-2}),$$
where $|\dM(L)|$ denotes the measure of $\dM(L)$. And we can deduce from above that
$$R^*=
\sum_{q\le L}\sum_{\substack{1\le a\le q\\
		(a,q)=1}}
\int_{|\alpha-\frac aq|\le\frac{L}{qn^2}}f^*(\alpha)\e\big(-\alpha(2m+n)\big)\,d\alpha
=\sum_{q\le L}T(q;m)\cdot\dJ^*(m),
$$
where $\dJ^*(m)=\int_{|\beta|\le\frac{L}{qn^2}}I(\beta)\e\big(-\beta(2m+n)\big)\, d\beta=
\dJ(m)+O\big(\frac{qn^2}{L}\big)$ using the definition of $\dJ(m)$ and $I(\beta)\ll |\beta|^{-2}.$

It is well-known (see Theorem 6.25 \cite{L96}) that for any $\epsilon>0,$
there exists $c_\epsilon>0$ such that
$\tau(n)\le c_\epsilon n^{\epsilon},$
where $\tau(n)$ denotes the number of factors of $n$.
Let $\Delta=n^2-5(2m+n)$ and $q=2^a5^bq_1$ with $a,b\ge 0$ and $(q_1,10)=1.$
By Claims \ref{Clm:fuction q}, \ref{Clm:general T(q,m)} and \ref{Clm:T(q,m),(q,10)=1},
we have
$q|T(q;m)|\le \frac{4(q_1,\Delta)}{q_1}.$
Thus
$\sum\limits_{q\le L}T(q;m)O\big(\frac{qn^2}{L}\big)\le
\sum\limits_{q\le L}\frac{(q,\Delta)}{q}O\Big(\frac{(\log L)^2n^2}{L}\Big).$
Since
$$\sum\limits_{q\le L}\frac{1}{q}(q,\Delta)=
\sum\limits_{\substack{d\mid \Delta\\
		d\le L}}
\sum\limits_{\substack{q\le L\\
			(q,\Delta)=d}}\frac dq\le
\sum\limits_{\substack{d\mid \Delta\\
		d\le L}}
\sum\limits_{k\le\frac Ld}\frac 1k\le	
\log L\cdot\sum\limits_{\substack{d\mid \Delta\\
		d\le L}}1\le c_\epsilon \Delta^{\epsilon}\log L=O\Big((\log L) n^{2\epsilon}\Big),$$
we can get that
$\sum\limits_{q\le L}T(q;m)O\big(\frac{qn^2}{L}\big)=O\Big(\frac{(\log L)^3n^{2+2\epsilon}}{L}\Big).$
Since $\sum\limits_{q\le L}T(q;m)=G(m)-\sum\limits_{q> L}T(q;m)$ and
$\dJ(m)=O(n^2),$
by the same method above,
we can get that
$\sum\limits_{q> L}T(q;m)\dJ(m)=O\Big(\frac{(\log L)^3n^{2+2\epsilon}}{L}\Big).$
Combining the above two bounds, we get that
$
R^*=G(m)\mathfrak{I}(m)+O\Big(\frac{(\log L)^3n^{2+2\epsilon}}{L}\Big).
$
By choosing
$\epsilon=\frac{1}{400},$
we finish the proof of Lemma \ref{Lem:M estimation}.
\end{proof}

Now we are ready to prove Theorem \ref{Thm:E(n)}, which would complete the proof of Theorem~\ref{Thm:second thm}.
\medskip

\noindent\textbf{Proof of Theorem \ref{Thm:E(n)}.}
Let $L=n^{\frac{1}{50}}.$
Let $W$ denote the set of integers $m\in \Big[\frac{n^2}{10}+\frac{n^2}{\log n}, \frac{n^2-n}{2}-\frac{n^2}{\log n}\Big]$ such that $\Big|\int_{\dm(L)}f(\alpha)\e\big(-\alpha (2m+n)\big)\, d\alpha\Big|\ge n^2(\log n)^{-11}.$
Since $L=n^{\frac{1}{50}}\le \frac n2,$
by Lemma \ref{Lem:m estimation},
we have
$$|W|=O\left(\frac{n^6(\log n)^5L^{-2}}{n^4(\log n)^{-22}}\right)=O\big(n^2(\log n)^{27}L^{-2}\big).$$
For integers $m\in \Big[\frac{n^2}{10}+\frac{n^2}{\log n}, \frac{n^2-n}{2}-\frac{n^2}{\log n}\Big]\backslash W$,
by Claim \ref{Clm:G(m)}, Lemma \ref{Lem:M estimation} and $\dJ(m)\gg n^2(\log n)^{-10},$
we have
\begin{align*}
\dR(m)&=\int_{\frac1n}^{1+\frac1n}f(\alpha)\e\big(-\alpha (2m+n)\big)\, d\alpha
=\int_{\dm(L)}f(\alpha)\e\big(-\alpha (2m+n)\big)+\int_{\dM(L)}f(\alpha)\e\big(-\alpha (2m+n)\big)\\
&=G(m)\dJ(m)+O(n^{2-\frac{1}{100}})+O(n^2(\log n)^{-11})
\gg n^2(\log n)^{-10}(\log\log n)^{-1}>0.
\end{align*}
In particular, we see that $\dE(n)\subseteq W$ and
$|\dE(n)|\le |W|=O(n^2(\log n)^{27}L^{-2})=O(n^{2-\frac{1}{50}}).$
This completes the proof of Theorem~\ref{Thm:E(n)}.\QED

\section{Concluding remarks}
We now discuss some open problems related to the study of induced subgraphs of given sizes here.

Erd\H{o}s, F\"uredi, Rothschild and S\'os \cite{EFRS} also conjectured that the limsup in \eqref{equ:sigma} is actually a limit.
That says, the limit $\sigma^*(m,f)=\lim_{n\rightarrow\infty}|S_n(m,f)|/\binom{n}{2}$ exists for every pair $(m,f)$.
As mentioned earlier, a result of \cite{EFRS} implies that the majority of pairs $(m,f)$ satisfy $\sigma(m,f)=0$,
and in these cases, it is clear that $\sigma^*(m,f)$ exists (for being zero).
For the pairs $(m,f)$ from Theorem~\ref{Thm:Non zeros} with $r=2$ or $r\geq 5$,
by Theorem~\ref{Thm:second thm} we have $\sigma(m,f)=\frac1r$, and in fact, it also holds that $\sigma^*(m,f)=\frac1r$.
Despite of these supportive results, it seems to be a challenging problem to determine the existence of this limit for any pair $(m,f)$ with $\sigma(m,f)>0$.
We prove the following lemma towards this conjecture, which says that it would suffice to find some large $S_n(m,f)$ with most of its integers appearing in few consecutive intervals.

\begin{lem}\label{Lem:Limit}
Let $(m,f)$ be a pair with $\sigma=\sigma(m,f)>0$.
If for any $\epsilon>0,$ there exists some integer $n$ such that
at least $(\sigma-\epsilon)\binom{n}{2}$ integers in $S_{n}(m,f)$ belong to a union of at most $\epsilon \sqrt{n}$ intervals of consecutive integers,
then the limit $\sigma^*(m,f)$ exists.
\end{lem}

\begin{proof}
Fix any $\epsilon>0$. Then there exists some $n$ such that
$S_{n}(m,f)$ contains disjoint intervals $I_1, I_2, ..., I_k$ of consecutive integers, where $k\leq \epsilon \sqrt{n}$ and $\sum_{j=1}^k |I_j|\geq (\sigma-\epsilon)\binom{n}{2}$.
Let $I_j=\big[c_j\binom{n}{2},d_j\binom {n}2\big]$ for each $j\in [k]$, where $\sum_{j=1}^k(d_j-c_j)\geq \sigma-\epsilon.$
Let $N$ be any sufficiently large integer and let $\mathcal{E}=\bigcup_{j=1}^kI_j',$ where
$I_j'=\Big[\big(c_j+2n^{-\frac 12}\big)\binom{N}{2},\big(d_j-2n^{-\frac 12}\big)\binom {N} 2\Big]$ for $j\in [k].$

We claim that $\mathcal{E}\subseteq S_N(m,f)$.
To see this, consider any $E\in \mathcal{E}$ (say $E\in I_j'$ for some $j\in [k]$) and any $N$-vertex graph $G$ with $E$ edges.
We aim to show that $G$ contains an $n$-vertex induced subgraph with $e$ edges, where $e\in I_j$.
For an $n$-subset $A$ of $V(G)$, let $e(A)$ be the number of edges in the induced subgraph $G[A].$
Then for any $A,A'\in\binom{V(G)}{n}$ with $|A\cap A'|=n-1,$ we can easily check that $|e(A)-e(A')|\le n-1.$
Let $C$ be a uniformly random element of $\binom{V(G)}{n}.$
It is easy to see that $\mathbb{E}[e(C)]=E\cdot\binom{N-2}{n-2}/\binom{N}{n}=\frac{E}{\binom N 2}\binom {n} 2.$
By Lemma~\ref{Lem:Concentration} (with $\alpha=n-1$ and $t=n^\frac12(n-1)$), we get that
$$P\left(\Big|e(C)-\mathbb{E}[e(C)]\Big|\ge n^\frac12(n-1) \right)\le 2\exp\Big(-\frac{2n(n-1)^2}{\min\{n,N-n\}(n-1)^2}\Big)\le 2e^{-2}<1.$$
Therefore, there must exist an $n$-subset $B$ of $V(G)$ such that
$\Big|e(B)-\frac{E}{\binom N 2}\binom {n} 2\Big|< n^\frac12(n-1)$.
This gives that
$$c_j\binom {n} 2\leq \frac{E}{\binom N 2}\binom {n} 2-n^\frac12(n-1)\leq e(B)\le\frac{E}{\binom N 2}\binom {n} 2+n^\frac12(n-1)\leq d_j\binom {n} 2,$$
where $E\in I_j'=\Big[\big(c_j+2n^{-\frac 12}\big)\binom{N}{2},\big(d_j-2n^{-\frac 12}\big)\binom {N} 2\Big]$.
So $G[B]$ is an $n$-vertex induced subgraph of $G$, where $e(B)\in I_j\subseteq S_{n}(m,f).$
Then $G[B]$ (and thus $G$) contains an $m$-vertex induced subgraph with $f$ edges.
This shows that $E\in S_N(m,f)$ and thus $\mathcal{E}\subseteq S_N(m,f)$, proving the claim.

Since $\sum_{j=1}^k(d_j-c_j)\geq \sigma-\epsilon$ and $k\leq \epsilon \sqrt{n}$,
we see that for sufficiently large $N$,
$$|S_N(m,f)|\geq |\mathcal{E}|=\sum_{j=1}^{k}|I_j'|=\sum_{j=1}^{k}\Big(d_j-c_j-4n^{-\frac 12}\Big)\binom N2\ge\Big(\sigma-\epsilon-4kn^{-\frac 12}\Big)\binom N2\ge \big(\sigma-5\epsilon\big)\binom N2.$$
This shows that for any $\epsilon>0$, we have
$\liminf\limits_{N\rightarrow\infty}\frac{|S_N(m,f)|}{\binom N2}\ge \sigma-5\epsilon.$
Hence $\liminf\limits_{N\rightarrow\infty}\frac{|S_N(m,f)|}{\binom N2}\ge \sigma=\limsup\limits_{N\rightarrow\infty}\frac{|S_N(m,f)|}{\binom N2}$,
which implies that $\sigma^*(m,f)=\lim\limits_{N\rightarrow\infty}\frac{|S_N(m,f)|}{\binom N2}$ exists.
This completes the proof of Lemma \ref{Lem:Limit}.
\end{proof}

It also seems natural to consider the same problem for hypergraphs. Let $r\geq 3$ and $m, f$ be integers satisfying $0\leq f\leq \binom{m}{r}$.
Let $S_n^r(m,f)$ consist of all integers $e$ such that every $n$-vertex $r$-graph with $e$ edges contains an $m$-vertex induced subhypergraph with $f$ edges,
and let $\sigma_r(m,f)=\limsup_{n\rightarrow\infty}|S_n^r(m,f)|/\binom{n}{r}.$
Compared with the graph case (i.e., Theorem~\ref{Thm:main}), it is much easier to show the following.
\begin{lem}\label{Lem:sigma_r(m,f)}
Let $r\ge 3.$ If $(m,f)\notin\{(r,0), (r,1)\},$ then $\sigma_r(m,f)\le 1-r!/r^r;$ otherwise, $\sigma_r(m,f)=1.$
\end{lem}
\noindent In fact, it is enough to use the construction of complete $r$-partite $r$-graphs for all but finitely many pairs $(m,f)$.

\medskip

The authors of \cite{EFRS} mentioned the question if $\lim_{n\rightarrow\infty}\Big(\max_{0\leq f\leq \binom{m}{2}} \sigma(m,f)\Big)=0$.
It was not, as answered by their Theorem~\ref{Thm:Non zeros}.
A direct corollary of Theorem~\ref{Thm:main} shows that
$\limsup_{m\rightarrow\infty}\Big(\max_{0\le f\le \binom{m}{2}}\sigma(m,f)\Big)=\frac12,$
but it is not known yet if $\lim_{m\rightarrow\infty}\Big(\max_{0\le f\le \binom{m}{2}}\sigma(m,f)\Big)$ exists or not.

\medskip

We conclude this paper with a remark that by a more careful refinement of the arguments in Subsection~\ref{subsec:4.2},
one can show that $|C(n,r)|=\frac{n^2}{2}-\frac{n^2}{2r}+o(n^2)$ holds for $r=4$ and {\it odd} integers $n.$
It would be interesting to understand the remaining cases $r=3,4$ in Conjecture~\ref{Conj:equzlity}. 



\bibliographystyle{unsrt}

\end{document}